\documentclass[a4paper,12pt]{article}
\usepackage[T1]{fontenc}
\usepackage[utf8]{inputenc}
\usepackage{lmodern}
\usepackage{amsfonts}
\usepackage{amsmath}
\usepackage{latexsym}
\usepackage{amssymb}
\usepackage{graphicx}
\usepackage{wrapfig}
\usepackage{amsthm}
\usepackage{float}
\usepackage{epsfig}
\usepackage{multirow}
\usepackage[toc,page]{appendix}
\usepackage[footnotesize]{caption}
\RequirePackage{color}
\newsavebox{\mybox}

\usepackage{amscd,enumerate}
\usepackage{amsfonts}
\usepackage{showlabels}

\input xy
\xyoption{all}

\usepackage{multicol}
\usepackage{hyperref}
\hypersetup{ colorlinks,
linkcolor=blue,
filecolor=green,
urlcolor=blue,
citecolor=blue }

\newtheorem{lemma}{Lemma}[section]

\newtheorem{theorem}[lemma]{Theorem}
\newtheorem{proposition}[lemma]{Proposition}

\newtheorem{definition}[lemma]{Definition}

\newtheorem{example}[lemma]{Example}

\newtheorem{Elemento V.9}[lemma]{Elemento V.9}

\newcommand{\uloopr}[1]{\ar@'{@+{[0,0]+(-4,5)}@+{[0,0]+(0,10)}@+{[0,0] +(4,5)}}^{#1}}
\newcommand{\uloopd}[1]{\ar@'{@+{[0,0]+(5,4)}@+{[0,0]+(10,0)}@+{[0,0]+
(5,-4)}}^{#1}}
\newcommand{\dloopr}[1]{\ar@'{@+{[0,0]+(-4,-5)}@+{[0,0]+(0,-10)}@+{[0,
0]+(4,-5)}}_{#1}}
\newcommand{\dloopd}[1]{\ar@'{@+{[0,0]+(-5,4)}@+{[0,0]+(-10,0)}@+{[0,0
]+(-5,-4)}}_{#1}}
\newcommand{\luloop}[1]{\ar@'{@+{[0,0]+(-8,2)}@+{[0,0]+(-10,10)}@+{[0,
0]+(2,2)}}^{#1}}

\definecolor{turquoise2}{rgb}{0,0.898039,0.933333}
\definecolor{magenta}{rgb}{1,0,1}
\definecolor{olivedrab}{rgb}{0.419608,0.556863,0.137255}
\definecolor{purple2}{rgb}{0.568627,0.172549,0.933333}
\definecolor{amethyst}{rgb}{0.6, 0.4, 0.8}
\definecolor{ao(english)}{rgb}{0.0, 0.5, 0.0}
\definecolor{atomictangerine}{rgb}{1.0, 0.6, 0.4}
\definecolor{amber(sae/ece)}{rgb}{1.0, 0.49, 0.0}
\definecolor{alizarin}{rgb}{0.82, 0.1, 0.26}
\definecolor{auburn}{rgb}{0.43, 0.21, 0.1}
\definecolor{aqua}{rgb}{0.0, 1.0, 1.0}

\begin{document}

\title{The Local-Global Principle in Leavitt Path Algebras}
\author{Song\"{u}l ES\.{I}N \\
}
\maketitle





\begin{abstract}
This is a short note on how a particular graph construction on a subset of edges that lead to a subalgebra construction, provided a tool in proving some ring theoretical properties of Leavitt path algebras. 

\end{abstract}

\section{Introduction}
This paper is an expository note publicizing how a particular subalgebra construction 
which first appeared in the paper \cite{RegCond} by G. Abrams and K.M. Rangaswamy was used in proving many theorems on Leavitt path 
algebras. The power of the subalgebra construction relies on extending a particular property on a Leavitt path algebra over a 
"smaller" graph to the Leavitt path algebra of the whole graph. This can be visualised as from a local view to a global setting,  
"local-to-global jump". 

We start by recalling the definitions of a path algebra and a Leavitt path
algebra, (see \cite{AAS} for a more extended study on Leavitt path algebras). 
A \textit{directed graph} $E=(E^{0},E^{1},r,s)$
consists of two countable sets $E^{0},E^{1}$ and functions 
$r,s:E^{1}\rightarrow E^{0}$. The elements $E^{0}$ and $E^{1}$ are called 
\textit{vertices} and \textit{edges}, respectively. For each $e\in E^{0}$, 
$s(e)$ is the source of $e$ and $r(e)$ is the range of $e.$ If $s(e)=v$ and 
$r(e)=w$, then we say that $v$ emits $e$ and that $w$ receives $e$. A vertex
which does not receive any edges is called a \textit{source,} and a vertex
which emits no edges is called a \textit{sink.} A graph is called \textit{
row- finite} if $s^{-1}(v)$ is a finite set for each vertex $v$. For a
row-finite graph the edge set $E^{1}$ of $E~$is finite if its set of
vertices $E^{0}$ is finite. Thus, a row-finite graph is finite if $E^{0}$ is
a finite set.

A path in a graph $E$ is a sequence of edges $\mu =e_{1}\ldots e_{n}$ such
that $r(e_{i})=s(e_{i+1})$ for $i=1,\ldots ,n-1.$ In such a case, $s(\mu
):=s(e_{1})$ is the \textit{source }of $\mu $ and $r(\mu ):=r(e_{n})$ is the 
\textit{range} of $\mu $, and $n$ is the \textit{length }of $\mu ,$ i.e., 
$l(\mu )=n.$

If $s(\mu )=r(\mu )$ and $s(e_{i})\neq s(e_{j})$ for every $i\neq j$, then 
$\mu $ is called a \textit{cycle}. If $E$ does not contain any cycles, $E$ is
called \textit{acyclic}. For $n\geq 2,$ define $E^{n}$ to be the set of paths of length $n,$ and 
$E^{\ast }=\bigcup\limits_{n\geq 0}E^{n}$ the set of all finite paths.
Denote by $E^{\infty }$ the set of all infinite
paths of $E$, and by $E^{\leq \infty }$ the set $E^{\infty }$ together with
the set of finite paths in $E$ whose range vertex is a sink.  We say that a
vertex $v\in E^{0}$ is \textit{cofinal} if for every $%
\gamma \in E^{\leq \infty }$ there is a vertex $w$ in the path $\gamma $
such that $v\geq w$. We say that a graph $E$ is cofinal if every vertex in $%
E $ is cofinal.

The path $K$-algebra over $E$ is defined as the free $K$-algebra $
K[E^{0}\cup E^{1}]$ with the relations:

\begin{enumerate}
\item[(1)] $v_{i}v_{j}=\delta _{ij}v_{i}$ \ for every $v_{i},v_{j}\in E^{0}.$

\item[(2)] $e_{i}=e_{i}r(e_{i})=s(e_{i})e_{i}$ $\ $for every $e_{i}\in
E^{1}.$
\end{enumerate}

This algebra is denoted by $KE$. Given a graph $E,$ define the extended
graph of $E$ as the new graph $\widehat{E}=(E^{0},E^{1}\cup (E^{1})^{\ast
},r^{\prime },s^{\prime })$ where $(E^{1})^{\ast }=\{e_{i}^{\ast
}~|~e_{i}\in E^{1}\}$ and the functions $r^{\prime }$ and $s^{\prime }$ are
defined as 
\begin{equation*}
r^{\prime }|_{E^{1}}=r,~~~~s^{\prime }|_{E^{1}}=s,~~~~r^{\prime
}(e_{i}^{\ast })=s(e_{i})~~~~~~\text{and~~~~~}s^{\prime }(e_{i}^{\ast
})=r(e_{i}).
\end{equation*}
The Leavitt path algebra of $E$ with coefficients in $K$ is defined as the
path algebra over the extended graph $\widehat{E},$ with relations:

\begin{enumerate}
\item[(CK1)] $e_{i}^{\ast }e_{j}=\delta _{ij}r(e_{j})$ \ for every $e_{j}\in
E^{1}$ and $e_{i}^{\ast }\in (E^{1})^{\ast }.$

\item[(CK2)] $v_{i}=\sum_{\{e_{j}\in
E^{1}~|~s(e_{j})=v_{i}\}}e_{j}e_{j}^{\ast }$ \ for every $v_{i}\in E^{0}$
which is not a sink.
\end{enumerate}

This algebra is denoted by $L_{K}(E)$. The conditions (CK1) and (CK2) are
called the Cuntz-Krieger relations. In particular condition (CK2) is the
Cuntz-Krieger relation at $v_{i}$. If $v_{i}$ is a sink, we do not have a
(CK2) relation at $v_{i}$. Note that the condition of row-finiteness is
needed in order to define the equation (CK2).

\bigskip 

Given a graph, we define a new graph built upon the given one that will be necessary for the subalgebra construction. The construction is based on an idea presented by Raeburn and Szyma\'{n}ski in \cite[Definition 1.1]{Raeburn}. Then, we construct several
examples.


\begin{definition} \cite[Definition 2]{RegCond} 
Let $E$ be a graph, and $F$ be a finite set of edges in $E.$ We define $s(F)$
(resp. $r(F)$) to be the sets of those vertices in $E$ which appear as the
source (resp. range) vertex of at least one element of $F.$ We define a
graph $E_{F}$ as follows: 
\begin{equation*}
E_{F}^{0}=F\cup (r(F)\cap s(F)\cap s(E^{1}\backslash F))\cup (r(F)\backslash
s(F)),
\end{equation*}
$$\begin{array}{rcl}
E_{F}^{1} & = & \{(e,f)\in F\times E_{F}^{0}~|~r(e)=s(f)\} \\
& \cup & \lbrack
\{(e,r(e))~|~e\in F\text{ with }r(e)\in (r(F)\backslash s(F))\} \rbrack,
\end{array} $$
and where $s((x,y))=x,$ $r((x,y))=y$ for any $(x,y)\in E_{F}^{1}.$
\end{definition}

\begin{example}\cite[Example 1]{RegCond} \rm
Let $E$ be the rose with $n$-petals graph
$$ 
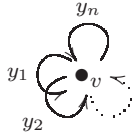
\xymatrix{
{\bullet}_v  \ar@(ul,dl)_{y_1} \ar@(l,d)_{y_2} \ar@{.>}@(d,r)  \ar@(ur,ul)_{y_n}    }
$$ 

Let $F=\{y_{1}\}$. Then $E_{F}^{0}=\{y_{1}\}\cup \{v\},$ and 
$E_{F}^{1}=\{(y_{1},y_{1}),(y_{1},v)\}$. Pictorially, $E_{F}$ is given by

$$ 
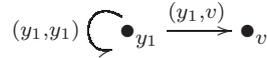
\xymatrix{
{\bullet}_{y_1}  \ar@(ul,dl)_{(y_1,y_1)}   \ar@{->}[r]^{(y_1,v)}  & {\bullet}_v  }
$$ 

This example indicates that various properties of the graph $E$ need not
pass to the graph $E_{F}.$ For instance, $E$ is cofinal, while $E_{F}$ is
not. In particular, $L_{K}(E)$ is a simple algebra, while $L_{K}(E_{F})$ is
not.
\end{example}

\begin{example} \rm Let $E$ be the graph
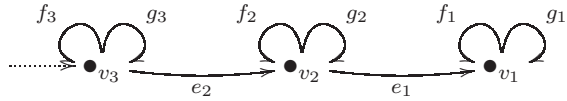
\[ 
\xymatrix{ \ar@{.>}[r] &
\bullet_{v_3}\ar@(u,l)_{f_3} \ar@(u,r)^{g_3} \ar@/_.3pc/[rr]_{e_2} & &  
\bullet_{v_2}\ar@(u,l)_{f_2} \ar@(u,r)^{g_2} \ar@/_.3pc/[rr]_{e_1}  && \bullet_{v_1}\ar@(u,l)_{f_1} \ar@(u,r)^{g_1} }
\]
and $F=\{f_{1},g_{1}\}$. Then, $E_{F}$ is given by 
$$ 
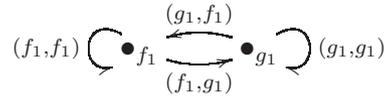
\xymatrix{
{\bullet}_{f_1}  \ar@(ul,dl)_{(f_1,f_1)}   \ar@/_.5pc/[r]_{(f_1,g_1)}  & {\bullet}_{g_1}  \ar@(ur,dr)^{(g_1,g_1)} \ar@/_.5pc/[l]_{(g_1,f_1)} }
$$

In this example $E$ is not cofinal but 
$E_{F}$ is cofinal. Also, $L_{K}(E)$ is not purely infinite simple while 
$L_{K}(E_{F})$ is.
\end{example}

\begin{example} \rm
Consider the infinite clock graph $E$ with one source which emits countably many edges as follows:
$$ 
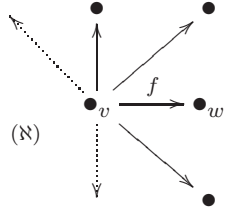
\xymatrix{&{\bullet} & {\bullet}\\
 &{\bullet}_v  \ar@{->}[dr] \ar@{->}[r]^f \ar@{->}[ur]  \ar@{->}[u]  \ar@{.>}[ul]  \ar@{.>}[d] \ar@{}[dl]_{(\aleph)} & {\bullet}_w
 \\  & &{\bullet} }
$$ 
Let $F=\{f\}$ and then $E_{F}$ is
$$ 
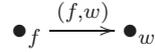
\xymatrix{
{\bullet}_{f}    \ar@{->}[r]^{(f,w)}  & {\bullet}_w  }
$$   
This is an example which shows that both $E$ and $E_{F}$ are acyclic graphs where $F$ is any subset of vertices. 
Actually, if $E$ is any acyclic graph and $F$ any subset of vertices then 
$E_{F}$ is acyclic is proved in \cite[Lemma 1]{RegCond}.
\end{example}

\section{The Subalgebra Construction}

Although in general $E_{F}$ need not be a subgraph of $E$, the Leavitt path
algebras $L_{K}(E_{F})$ and $L_{K}(E)$ are related via a homomorphism which
leads to a subalgebra construction of $L_{K}(E)$.

In \cite[Proposition 1]{RegCond}, for a finite set of edges $F$ in a graph $E$, 
the algebra homomorphism $\theta :L_{K}(E_{F})\rightarrow L_{K}(E)$ having
the properties

\begin{itemize}
\item[(1)] $F\cup F^{\ast }\subseteq \text{Im}(\theta ),$

\item[(2)] If $w\in r(F)$, then $w\in \text{Im}(\theta ),$

\item[(3)] If $w\in E^{0}$ has $s_{E}^{-1}(w)\subseteq F,$ then $w\in \text{Im}(\theta ),$
\end{itemize}

is defined by using the following subsets $G^{0}$ and $G^{1}$ of $L_{K}(E)$

$$\begin{array}{rcl}
G^{0} & = &\{ee^{\ast }~|~e\in F\}\cup \{v-\sum\limits_{f\in F,s(f)=v}ff^{\ast
}~|~v\in r(F)\cap s(F)\cap s(E^{1}\backslash F)\} \\
& \cup &  \{v~|~v\in r(F)\backslash s(F)\}
\end{array} $$
and 
\newpage
$$\begin{array}{rcl}
G^{1} &= &\{eff^{\ast }~|~e,f\in F,s(f)=r(e)\} \\
& \cup & \{e-\sum\limits_{f\in
F,s(f)=r(e)}eff^{\ast }~|~r(e)\in r(F)\cap s(F)\cap s(E^{1}\backslash F) \\
& \cup & \{e\in F~|~r(E)\in r(F)\backslash s(F)\}
\end{array} $$
\bigskip

In particular, $\theta (w)\in G^{0}$ for all vertices in $E_{F}$ and $\theta
(w)\in G^{1}$ for all edges in $E_{F}.$

\bigskip

Let $E$ be any graph, $K$ any field, and $\{a_{1},a_{2},\ldots ,a_{l}\}$ any
finite subset of nonzero elements of $L_{K}(E).$ For each $1\leq r\leq l$
write
\begin{equation*}
a_{r}=k_{c_{1}}v_{c_{1}}+k_{c_{2}}v_{c_{2}}+\ldots
+k_{c_{j(r)}}v_{c_{j(r)}}+\sum%
\limits_{i=1}^{t(r)}k_{r_{i}}p_{r_{i}}q_{r_{i}}^{\ast }
\end{equation*}
where each $k_{j}$ is a nonzero element of $K$, and , for each $1\leq i\leq
t(r),$ at least one of $p_{r_{i}}$ or $q_{r_{i}}$ has length at least $1.$
Let $\ F$ be denote the (necessarily finite) set of those edges in $E$ which
appear in the representation of some $p_{r_{i}}$ or $q_{r_{i}},$ $1\leq
r_{i}\leq t(r),~1\leq r\leq l.$ Now consider the set 
\begin{equation*}
S=\{v_{c_{1}},v_{c_{2}},\ldots ,v_{c_{j(r)}}~|~1\leq r\leq l\}
\end{equation*}
of vertices which appear in the displayed description of $a_{r}$ for some 
$1\leq r\leq l.$ We partition $S$ into subsets as follows:
\begin{equation*}
S_{1}=S\cap r(F),
\end{equation*}
and, for remaining vertices $T=S\backslash S_{1}$, we define 
\begin{eqnarray*}
S_{2} &=&\{v\in T~|~s_{E}^{-1}(v)\subseteq F\text{ and }s_{E}^{-1}(v)\neq
\emptyset \} \\
S_{3} &=&\{v\in T~|~s_{E}^{-1}(v)\cap F=\emptyset \} \\
S_{4} &=&\{v\in T~|~s_{E}^{-1}(v)\cap F\neq \emptyset \text{ and }
s_{E}^{-1}(v)\cap (E^{1}\backslash F)\neq \emptyset \}.
\end{eqnarray*}

\begin{definition}\cite[Definition 3]{RegCond}
Let $E$ be any graph, $K$ any field, and \newline
$\{a_{1},a_{2},\ldots ,a_{l}\}$ any finite subset of nonzero elements of 
$L_{K}(E).$ Consider the notation presented in The Subalgebra Construction.
We define $B(a_{1},a_{2},\ldots ,a_{l})$ to be the $K$-subalgebra of 
$L_{K}(E)$ generated by the set $\text{Im}(\theta )\cup S_{3}\cup S_{4}$.
That is,
\begin{equation*}
B(a_{1},a_{2},\ldots ,a_{l})=<\text{Im}(\theta ),S_{3},S_{4}>.
\end{equation*}
\end{definition}

\begin{proposition}\cite[Proposition 1]{RegCond}
\label{Prop1, 2}Let $E$ be any graph, $K$ any field, and 
$\{a_{1},a_{2},\ldots , a_{l}\}$ any finite subset of nonzero elements of 
$L_{K}(E)$. Let $F$ denote the subset of $E^{1}$ presented in The Subalgebra
Construction. For $w\in S_{4}$ let $u_{w}$ denote the element 
$w-\sum\limits_{f\in F,s(f)=w}ff^{\ast }$ of $L_{K}(E).$ Then

\begin{itemize}
\item[(1)] $\{a_{1},a_{2},\ldots ,a_{l}\}\subseteq B(a_{1},a_{2},\ldots
,a_{l}).$

\item[(2)] $B(a_{1},a_{2},\ldots ,a_{l})=\text{Im}(\theta )\oplus (\oplus
_{v_{i}\in S_{3}}Kv_{i})\oplus (\oplus _{w_{j}\in S_{4}}Ku_{w_{j}}).$

\item[(3)] The collection $\{B(S)~|~S\subseteq L_{K}(E),$ $S$ finite$\}$ is
an upward directed set of subalgebras of $L_{K}(E).$

\item[(4)] $L_{K}(E)=\underrightarrow{\lim }_{\{S\subseteq
L_{K}(E),~S~finite\}}B(S)$.
\end{itemize}
\end{proposition}

\bigskip

Proposition \ref{Prop1, 2}, can be modified to include some more properties of the subalgebra construction in \cite{RegCond}. 
For instance, the morphism $\theta $ in the construction is actually a graded morphism whose image is a
graded submodule of $L_K(E)$ and it also reveals some properties of cycles.

The stronger version of Proposition \ref{Prop1, 2} is given in \cite
{HazratRanga} as Theorem 4.1

\begin{theorem}\cite[Theorem 4.1]{HazratRanga}
\label{Thm4.1, 3}For an arbitrary graph $E$, the Leavitt path
algebra $L_{K}(E)$ is a directed union of graded subalgebras $B=A\oplus
K\epsilon _{1}\oplus \cdots \oplus K\epsilon _{n}$ where $A$ is the image of
a graded homomorphism $\theta$ from a Leavitt path algebra $L_{K}(F_{B})$
to $L_{K}(E)$ where $F_{B}$ a finite graph which depends on $B$, the elements $
\epsilon _{i}$ are homogeneous mutually orthogonal idempotents and $\oplus $
is a ring direct sum. Moreover, if $E$ is acyclic, so is each graph $F_{B}$
and in this case $\theta $ is a graded monomorphism.
\end{theorem}

Moreover, any cycle $c$ in the graph $F_{B}$ gives rise to a cycle $c'$ in $E$ such that if 
$c$ has an exit in $F_{B}$ then $c'$ has an
exit in $E.$ In particular, a cycle in $F_{B}$ is of the form 
$(f_{1},f_{2})(f_{2},f_{3})\ldots (f_{n},f_{1})$ and this case 
$f_{1}f_{2}\ldots f_{n}$ is a cycle in $E$.

\bigskip

Throughout recent literature this subalgebra construction has been a
powerful tool. The first theorem that appears in the literature is the
following:

\begin{theorem}\cite[Theorem 1]{RegCond}
$L_{K}(E)$ is von Neumann regular if and only if $E$ is acyclic. If $E$ is
acyclic, then $L_{K}(E)$ is locally $K$-matricial; that is, $L_{K}(E)$ is
the direct union of subrings, each of which is isomorphic to a finite matrix
rings over $K.$
\end{theorem}

Now, we give one implication of the statement to demonstrate how the
subalgebra construction is used in the proof:

\bigskip

\begin{proof}
We assume $E$ is acyclic. Let $\{B(S)~|~S\subseteq L_{K}(E),$ $S$ finite$\}$
be the collection of subalgebras of $L_{K}(E)$ indicated in Proposition \ref
{Prop1, 2}(3). By Proposition \ref{Prop1, 2}(4), it suffices to show that
each such $B(S)$ is of the indicated form. But by Proposition \ref{Prop1, 2}
(2), $B(S)=B(a_{1},a_{2},\ldots ,a_{l})=\text{Im}(\theta )\oplus (\oplus
_{v_{i}\in S_{3}}Kv_{i})\oplus (\oplus _{w_{j}\in S_{4}}Ku_{w_{j}}).$ Since
terms appearing in the second and third summands are clearly isomorphic as
algebras to $K\cong M_{1}(K),$ it suffices to show that $\text{Im}(\theta )$
is isomorphic to a finite direct sum of finite matrix rings over $K.$ Since 
$E$ is acyclic, by Lemma 1 in \cite{RegCond} we have that $E_{F}$ is acyclic.
But $E_{F}$ is always finite by definition, so we have by \cite[Proposition
3.5]{APS}, that $L_{K}(E_{F})\cong \oplus _{i=1}^{l}M_{m_{i}}(K)$ for some $
m_{1},\ldots ,m_{l}$ in $\mathbb{N}$. Since each $M_{m_{i}}(K)$ is a simple ring, we have that any homomorphic
image of $L_{K}(E_{F})$ must have this same form. So we get that $\text{Im}
(\theta )\cong \oplus _{i=1}^{t}M_{m_{i}}(K)$ for some $m_{1},\ldots ,m_{t}$
in $\mathbb{N}$, and we are done. (As remarked previously, since $\theta $ is in fact an
isomorphism we have $t=l.$)
\end{proof}

\bigskip

We list the following theorems which are using the same Subalgebra
Construction in their proofs. In particular, we only quote the parts that
uses the Subalgebra Construction.

\bigskip

\begin{theorem}\cite[Theorem 5.1]{HazratRanga}
Let $E$ be an arbitrary graph. Then for the Leavitt path algebra $L_{K}(E)$
the following are equivalent:

\begin{itemize}
\item[(1)] Every left/right ideal of $L_{K}(E)$ is graded;

\item[(2)] The class of all simple left/right $L_{K}(E)$-modules coincides
with the class of all graded-simple left/right $L_{K}(E)$-modules;

\item[(3)] The graph $E$ is acyclic.
\end{itemize}
\end{theorem}

\begin{proof}
$(3)\Rightarrow (1)$ For the sake of simplicity of the notation, let $L:=L_{K}(E)$. 
Suppose $E$ is acyclic. Now, by Theorem \ref{Thm4.1, 3}, $L$
is a direct union of graded subalgebras $B_{\lambda }$ where $\lambda \in I,$
an index set and where each $B_{\lambda }$ is a finite direct sum of copies
of $K$ and a graded homomorphic image of a Leavitt path algebra of a finite
acyclic graph. By \cite[Theorem 4.14]{Hazrat2}, Leavitt path algebras of
finite acyclic graphs are semi-simple algebras which have elementary
gradings, that is, all the matrix units are homogeneous. Consequently, every
ideal of each $B_{\lambda }$ is graded. Let 
$L=\bigoplus\limits_{n\in\mathbb{Z}}L_{n}$ be the $\mathbb{Z}$-graded decomposition of $L.$ 
Since the $B_{\lambda }$ are graded
subalgebras, each $B_{\lambda }=\bigoplus\limits_{n\in \mathbb{Z}}(B_{\lambda }\cap L_{n}).$ 
Let $M$ be a left ideal of $L$. To show that $M$
is graded, we need only to show that $M=\bigoplus\limits_{n\in \mathbb{Z}}(M\cap L_{n}).$ 
Let $a\in M.$ Then, for some $\lambda ,$ $a\in M\cap
B_{\lambda }.$ Note that $M\cap B_{\lambda }=B_{\lambda }$ or a left ideal
of $B_{\lambda }.$ Since every left ideal of $B_{\lambda }$ and in
particular $M\cap B_{\lambda }$ is graded, we can write $a=a_{n_{1}}+\cdots
+a_{n_{k}}$ where
\begin{equation*}
a_{n_{i}}\subset (M\cap B_{\lambda })\cap (B_{\lambda }\cap
L_{n_{i}})\subset M\cap L_{n_{i}}
\end{equation*}
for $i=1,\ldots ,k.$ This show that $M=\bigoplus\limits_{n\in\mathbb{Z}}(M\cap L_{n})$ and 
hence $M$ is a graded left ideal of $L.$
\end{proof}

\bigskip

The next result is about graded von Neumann regular Leavitt path algebras. A
ring $R$ is von Neumann regular if for every $x\in R$ there exists $y\in R$ such
that $x=xyx.$ Moreover, a graded ring $R$ is graded von Neumann regular if each homogeneous element is von Neumann regular.

\bigskip

\begin{theorem}\cite[Theorem 4.2]{HazratRanga}; \cite[Theorem 10]{Hazrat}
Every Leavitt path algebra $L_{K}(E)$ of an arbitrary graph $E$ is a graded
von Neumann regular ring.
\end{theorem}

\begin{proof}\cite[Proof of Theorem 4.2]{HazratRanga}
Suppose $E$ is an arbitrary graph. By \cite[Theorem 4.1]{HazratRanga}, $
L_{K}(E)$ is a directed union of graded subalgebras $B=A\oplus K\epsilon
_{1}\oplus \cdots \oplus K\epsilon _{n}$ where $A$ is the image of a graded
homomorphism $\theta $ from a Leavitt path algebra $L_{K}(F_{B})$ to $
L_{K}(E)$ with $F_{B}$ a finite graph (depending on $B$), the elements $
\epsilon _{i}$ are homogeneous mutually orthogonal idempotents and $\oplus $
is a ring direct sum. Since $F_{B}$ is a finite graph, $L_{K}(F_{B})$ and
hence $B$ is graded von Neumann regular by \cite{Hazrat}. It is then clear
from the definition that the direct union $L_{K}(E)$ is also graded von
Neumann regular.
\end{proof}

\bigskip

Recall that a ring $R$ is called left B\'{e}zout in case every finitely
generated left ideal of $R$ is principal. If the graph $E$ is finite, then 
$L_{K}(E)$ is B\'{e}zout \cite[Theorem 15]{Bezout}. 
The proof of this statement is given via a nice induction argument which we do not quote here. 
The generalization of this result to arbitrary graphs, which again appears in \cite{Bezout}, uses the subalgebra construction.
\bigskip

\begin{theorem}
	\cite[Corollary 16]{Bezout}
Let $E$ be an arbitrary graph and $K$ any field. Then $L_{K}(E)$ is B\'{e}zout.
\end{theorem}

\begin{proof}
By Theorem \ref{Thm4.1, 3}, $L_{K}(E)$ is the direct limit of unital
subalgebras, each of which is isomorphic to the Leavitt path $K$-algebra of
a finite graph. By \cite[Theorem 15]{Bezout}, each of these unital
subalgebras is a B\'{e}zout subring of $L_{K}(E)$. 

\medskip

Now, we are going to prove that for any ring $R$, if every finite subset of $R$ is contained in a unital B\'{e}zout subring of $R$, 
then $R$ is B\'{e}zout. Let us consider a finitely generated left ideal of $R$ with generators $x_1, x_2, \dots , x_n \in R$. 
Then there is a unital B\'{e}zout subring $S$ of $R$ that contains $\{x_1, x_2, \dots, x_n \}$. Hence, there exists $x \in S$ such that 
the left $S$-ideal $Sx_1 + Sx_2 + \cdots + Sx_n = Sx$.  
 
Since $1_{S}x_{i}=x_{i}$ for all $1\leq i\leq n$, and
each $x_{i}$ is in $Sx_{1}+Sx_{2}+\cdots +Sx_{n}=Sx$ which implies that for
each $i$ there exists $s_{i}\in S$ with $x_{i}=s_{i}x$. 

Hence $Rx_1+Rx_2+\cdots +Rx_n=Rs_{1}x+Rs_{2}x+\cdots +Rs_{n}x\subseteq Rx$. 
Also, $x=1_{s}x\in Sx$ implies $x\in Sx_{1}+Sx_{2}+\cdots +Sx_{n}\subseteq Rx_{1}+Rx_{2}+\cdots
+Rx_{n}$. Therefore, $Rx_1 + Rx_2 + \cdots + Rx_n = Rx$ and $R$ is a B\'{e}zout ring. 

\medskip 
Hence, if $R$ is taken to be $L_K(E)$, the result follows.
\end{proof}

\bigskip 

Recall that a ring with local units $R$ is said to be 
\textit{directly finite} if for every $x,y\in R$ and an idempotent element $
u\in R$ such that $xu=ux=x$ and $yu=uy=y$, we have that $xy=u$ implies $yx=u$.

\bigskip 

\begin{theorem}\cite[Proposition 4.3]{L.Vas}\label{Bora}
$L_{K}(E)$ is directly finite if and only if no cycle in $E$ has an exit. 
\end{theorem}

\bigskip 

The converse of Theorem \ref{Bora} for Leavitt path algebras of finite graphs has
been proven in \cite[Theorem 3.3]{L.Vas2}. To get the infinite graphs, Lia Vas proved the theorem by using Cohn-Leavitt approach.
In particular, the localization of the graph is used by considering a finite subgraph generated
by the vertices and edges of just those paths that appear in representations
of $x,$ $y$ and $u$ in $L_{K}(E)$ where $xy=u$ for some local unit $u$. However, the subgraph $F$ 
defined in this way may not produce a subalgebra $L_{K}(F)$ of $L_{K}(E)$. This problem is avoided by 
considering  an appropriate finite subgraph $F$ such
that the Cohn-Leavitt algebra of $F$ is a subalgebra of $L_{K}(E)$ and then adapts \cite[Theorem 3.3]{L.Vas2} 
to Cohn-Leavitt algebras of finite graphs.

\bigskip 

An alternative proof using the subalgebra construction is pointed out in \cite[Theorem 3.7]{HRS} using the grading on matrices. We outline the proof below (without considering the grading to refer to Theorem \ref{Bora}).
\begin{theorem}(\cite[Theorem 3.7]{HRS} rephrased)
For an arbitrary graph $E$, the following properties are equivalent for $L_K(E)$:
\begin{itemize}
\item[(a)] No cycle in $E$ has an exit;
\item[(b)]$L_K(E)$ is a directed union of graded semisimple Leavitt path algebras; specifically, $L_K(E)$ is a directed union of
direct sums of matrices of finite order over $K$ or $K[x, x^{-1}]$.
\item[(c)] $L_K(E)$ is directly-finite.
\end{itemize}
\end{theorem}
\begin{proof}
(a) implies (b) 
Assume (a). By Theorem \ref{Thm4.1, 3}, $L_K(E)$ is a directed union of graded subalgebras 
$B=A\oplus K\epsilon _{1}\oplus \cdots \oplus K\epsilon _{n}$, where $A$ is the image of a graded homomorphism 
$\theta$ from a Leavitt path algebra $L_K(F_B)$ to $L_K(E)$ with $F_B$ a finite
graph depending on $B$. Moreover, any cycle with an exit in $F_B$ gives rise to a cycle with an exit in $E$.
Since no cycle in $E$ has an exit, no cycle in the finite graph $F_B$ has an exit. So by using \cite[Theorem 2.7.3]{AAS},
$$L_K(F_B) \cong \bigoplus_{i \in I} M_{n_i}(K) \oplus \bigoplus_{j \in J} M_{m_j}(K[x, x^{-1}]),$$ where $n_i$ and $m_j$ are positive integers
 $I$, $J$ are index sets. Since the matrix rings $M_{n_i}(K)$ and $M_{m_j}(K[x, x^{-1}])$  are simple rings, $A$ and hence $B$ is a direct sum of finitely many matrix rings of finite order over $K$ and/or $K[x,x^{-1}]$. This proves (b).

(b) implies (c) follows from the known fact that matrix rings $M_{n_i}(K)$ and $M_{m_j}(K[x, x^{-1}])$ are
directly-finite and finite ring direct sums of such matrix rings are directly-finite. 
Hence, by condition (b), $L_K(E)$ is directly-finite.
\end{proof}

\bigskip

We want to finish the survey with another application of the Subalgebra
Construction. In \cite{GonBroxMer}, the authors do not use the exact
results, however they carry the same techniques and proofs to another
subgraph (dual graph) construction.

The authors present the notion of a dual of a subgraph in a graph, which is
the generalization of the usual notion of dual graph found in the literature
that we quote here:

\bigskip

\textbf{Usual dual:} Let $E$ be an arbitrary graph. The \textit{usual dual}
of $E$, $D(E),$ is the graph formed from $E$ by taking 
\begin{eqnarray*}
D(E)^{0} &=&\{e~|~e\in E^{1}\} \\
D(E)^{1} &=&\{ef~|~ef\in E^{2}\} \\
s_{D(E)}(ef) &=&e,~~r_{D(E)}(ef)=f\text{ \ for all }ef\in E^{2}.
\end{eqnarray*}%
The interest on the usual dual graph notion in the context of Leavitt path
algebras lies on the fact that, if $E$ is a row-finite graph without sinks,
then there is an algebra isomorphism $L_{K}(E)\cong L_{K}(D(E))$ 
(\cite[Proposition 2.11]{AALP}). These statement is untrue for usual dual of a graph with
sinks. The authors propose a new definition of dual graph which generalizes
this important property to row-finite graphs with sinks.

Dual of $F$ in $E$: Let $E$ be a graph and let $F$ be a subgraph of $E.$
Denote $F_{1}^{0}=\{v\in F^{0}~|~s_{F}^{-1}(v)=\emptyset
\},~F_{1}^{1}=r_{F}^{-1}(F_{1}^{0})$ and $F_{2}^{0}=s(F^{1})\cap
s(E^{1}\backslash F^{1}),~F_{2}^{1}=r_{F}^{-1}(F_{2}^{0}).$ The graph $
D_{E}(F),$ the \textit{dual of }$F$ \textit{in} $E$ is defined by
\begin{eqnarray*}
D_{E}(F)^{0} &=&D(F)^{0}\cup F_{1}^{0}\cup F_{2}^{0} \\
D_{E}(F)^{1} &=&D(F)^{1}\cup F_{1}^{1}\cup F_{2}^{1} \\
s_{D_{E}(F)}|_{D(F)} &=&s_{D(F)},~r_{D_{E}(F)}|_{D(F)}=r_{D(F)}
\end{eqnarray*}
For all $e\in F_{i}^{1}$ with $i\in \{1,2\},s_{D_{E}(F)}=e\in
D(F)^{0},~r_{D_{E}(F)}(e)=r_{F}(e)\in F_{i}^{0}.$

\bigskip

\textbf{Dual graph:} Given a graph $E$, they define $d(E)=D_{E}(E)$ and call
it the dual graph of $E$.

Then they prove the graded algebra isomorphism $L_{K}(d(E))\cong L_{K}(E)$
when $E$ is a row-finite graph (\cite[Proposition 3.6]{GonBroxMer}). In
this paper the authors also prove that for a graph $E$ and a row-finite
subgraph of $E$ there is a graded monomorphism 
$\theta:L_{K}(D_{E}(F))\rightarrow L_{K}(E)$. In addition, $F^{0}\cup
F^{1}\subseteq \theta (L_{K}(D_{E}(E))).$ This result is stated as \cite[Proposition 3.8]{GonBroxMer} 
and the proof is basically rephrasing \cite[Proposition 1,2]{RegCond}.

\end{document}